\documentclass[reqno]{amsart}
\usepackage{setspace,amssymb}
\usepackage{ifpdf}
\ifpdf
 \usepackage[hyperindex]{hyperref}
\else
 \expandafter\ifx\csname dvipdfm\endcsname\relax
 \usepackage[hypertex,hyperindex]{hyperref}
 \else
 \usepackage[dvipdfm,hyperindex]{hyperref}
 \fi
\fi
\allowdisplaybreaks[4]
\theoremstyle{plain}
\newtheorem{thm}{Theorem}[section]
\newtheorem{lem}{Lemma}[section]
\theoremstyle{remark}
\newtheorem{rem}{Remark}[section]
\DeclareMathOperator{\td}{d}
\newcommand{\bell}{\textup{B}}
\numberwithin{equation}{section}

\begin{document}

\title[Two closed form for the Bernoulli polynomials]
{Two closed forms for the Bernoulli polynomials}

\author[F. Qi]{Feng Qi}
\address[Qi]{Institute of Mathematics, Henan Polytechnic University, Jiaozuo City, Henan Province, 454010, China; College of Mathematics, Inner Mongolia University for Nationalities, Tongliao City, Inner Mongolia Autonomous Region, 028043, China; Department of Mathematics, College of Science, Tianjin Polytechnic University, Tianjin City, 300387, China}
\email{\href{mailto: F. Qi <qifeng618@gmail.com>}{qifeng618@gmail.com}, \href{mailto: F. Qi <qifeng618@hotmail.com>}{qifeng618@hotmail.com}, \href{mailto: F. Qi <qifeng618@qq.com>}{qifeng618@qq.com}}
\urladdr{\url{https://qifeng618.wordpress.com}}

\author[R. J. Chapman]{Robin J. Chapman}
\address[Chapman]{Mathematics Research Institute, University of Exeter, United Kingdom}
\email{\href{mailto: R. J. Chapman <R.J.Chapman@exeter.ac.uk>}{R.J.Chapman@exeter.ac.uk}, \href{mailto: R. J. Chapman <rjc@maths.ex.ac.uk>}{rjc@maths.ex.ac.uk}}
\urladdr{\url{www.maths.ex.ac.uk/~rjc/rjc.html}}
\urladdr{\url{http://empslocal.ex.ac.uk/people/staff/rjchapma/rjc.html}}

\begin{abstract}
In the paper, the authors find two closed forms involving the Stirling numbers of the second kind and in terms of a determinant of combinatorial numbers for the Bernoulli polynomials and numbers.
\end{abstract}

\keywords{closed form; Bernoulli polynomial; Bernoulli number; Stirling numbers of the second kind; determinant}

\subjclass[2010]{Primary 11B68, Secondary 11B73, 26A06, 26A09, 26C05}

\thanks{This paper was typeset using \AmS-\LaTeX}

\maketitle

\section{Introduction}

It is common knowledge that the Bernoulli numbers and polynomials $B_{k}$ and $B_k(u)$ for $k\ge0$ satisfy $B_k(0)=B_k$ and may be generated respectively by
\begin{equation*}
\frac{z}{e^z-1}=\sum_{k=0}^\infty B_k\frac{z^k}{k!}=1-\frac{z}2+\sum_{k=1}^\infty B_{2k}\frac{z^{2k}}{(2k)!}, \quad |z|<2\pi
\end{equation*}
and
\begin{equation*}
\frac{ze^{uz}}{e^z-1}=\sum_{k=0}^{\infty}B_{k}(u) \frac{z^k}{k!},\quad |z|<2\pi.
\end{equation*}
Because the function $\frac{x}{e^x-1}-1+\frac{x}2$ is odd in $x\in\mathbb{R}$, all of the Bernoulli numbers $B_{2k+1}$ for $k\in\mathbb{N}$ equal $0$. It is clear that $B_0=1$ and $B_1=-\frac12$. The first few Bernoulli numbers $B_{2k}$ are
\begin{align*}
B_2&=\frac16, & B_4&=-\frac1{30}, & B_6&=\frac1{42}, & B_8&=-\frac1{30}, \\
B_{10}&=\frac5{66}, & B_{12}&=-\frac{691}{2730}, & B_{14}&=\frac76, & B_{16}&=-\frac{3617}{510}.
\end{align*}
The first few Bernoulli polynomials are
\begin{gather*}
B_0(u)=1, \quad B_1(u)=u-\frac12, \quad B_2(u)=u^2-u+\frac16,\\
B_3(u)=u^3-\frac32u^2+\frac12u, \quad B_4(u)=u^4-2u^3+u^2-\frac1{30}.
\end{gather*}
\par
In combinatorics, the Stirling numbers $S(n,k)$ of the second kind for $n\ge k\ge1$ may be computed and generated by
\begin{equation*}
S(n,k)=\frac1{k!}\sum_{\ell=1}^k(-1)^{k-\ell}\binom{k}{\ell}\ell^{n}\quad\text{and}\quad
\frac{(e^x-1)^k}{k!}=\sum_{n=k}^\infty S(n,k)\frac{x^n}{n!}
\end{equation*}
respectively. See~\cite[p.~206]{Comtet-Combinatorics-74}.
\par
It is easy to see that the generating function
\begin{equation}\label{Bern-Polyn-Int-Expr}
\frac{ze^{uz}}{e^z-1}=\biggl[\frac{e^{(1-u)z}-e^{-uz}}z\biggr]^{-1}
=\frac1{\int_{-u}^{1-u}e^{zt}\td t} =\frac1{\int_0^1 e^{z(t-u)}\td t}.
\end{equation}
This expression will play important role in this paper. For related information on the integral expression~\eqref{Bern-Polyn-Int-Expr}, please refer to~\cite{emv-log-convex-simple.tex, ijmest-bernoulli, Guo-Qi-Filomat-2011-May-12.tex, jmaa-ii-97, (b^x-a^x)/x} and plenty of references cited in the survey and expository article~\cite{Qi-Springer-2012-Srivastava.tex}.
\par
In mathematics, a closed form is a mathematical expression that can be evaluated in a finite number of operations. It may contain constants, variables, four arithmetic operations, and elementary functions, but usually no limit.
\par
The main aim of this paper is to find two closed forms for the Bernoulli polynomials and numbers $B_k(u)$ and $B_k$ for $k\in\mathbb{N}$.
\par
The main results may be summarized as the following theorems.

\begin{thm}\label{BP-Stirl-thm}
The Bernoulli polynomials $B_n(u)$ for $n\in\mathbb{N}$ may be expressed as
\begin{multline}\label{BP-Stirl-form}
B_n(u)=\sum_{k=1}^nk! \sum_{r+s=k}\sum_{\ell+m=n}(-1)^m\binom{n}{\ell} \frac{\ell!}{(\ell+r)!}\frac{m!}{(m+s)!}\Biggl[\sum_{i=0}^r\sum_{j=0}^s(-1)^{i+j}\\
\times\binom{\ell+r}{r-i}\binom{m+s}{s-j}S(\ell+i,i)S(m+j,j)\Biggr]u^{m+s}(1-u)^{\ell+r}.
\end{multline}
Consequently, the Bernoulli numbers $B_k$ for $k\in\mathbb{N}$ can be represented as
\begin{equation}\label{Bernoulli-Stirling-formula}
B_n=\sum_{i=1}^n(-1)^{i}\frac{\binom{n+1}{i+1}}{\binom{n+i}{i}} S(n+i,i).
\end{equation}
\end{thm}

\begin{thm}\label{Bern-polyn-thm}
Under the conventions that $\binom{0}{0}=1$ and $\binom{p}{q}=0$ for $q>p\ge0$, the Bernoulli polynomials $B_k(u)$ for $k\in\mathbb{N}$ may be expressed as
\begin{equation}\label{Bern-Polyn-determ}
B_k(u)=(-1)^k\biggl|\frac1{\ell+1}\binom{\ell+1}{m} \bigl[(1-u)^{\ell-m+1}-(-u)^{\ell-m+1}\bigr]\biggr|_{1\le \ell\le k,0\le m\le k-1},
\end{equation}
where $|\cdot|_{1\le \ell\le k,0\le m\le k-1}$ denotes a $k\times k$ determinant.
Consequently, the Bernoulli numbers $B_k$ for $k\in\mathbb{N}$ can be represented as
\begin{equation}\label{Bern-No-determ}
B_k=(-1)^k\biggl|\frac1{\ell+1}\binom{\ell+1}{m}\biggr|_{1\le \ell\le k,0\le m\le k-1}.
\end{equation}
\end{thm}

\begin{rem}
The formula~\eqref{Bernoulli-Stirling-formula} recovers the one appeared in~\cite[p.~48, (11)]{Gould-MAA-1972}, \cite[(6)]{Guo-Qi-JANT-Bernoulli.tex}, \cite[p.~59]{Jeong-Kim-Son-JNT-2005}, and~\cite[p.~140]{Shirai-Sato-JNT-2001}. For detailed infirmation, please refer to~\cite[Remark~4]{Guo-Qi-JANT-Bernoulli.tex}. Hence, our Theorems~\ref{BP-Stirl-thm} and~\ref{Bern-polyn-thm} generalize those corresponding results obtained in~\cite{Gould-MAA-1972, Guo-Qi-JANT-Bernoulli.tex, Jeong-Kim-Son-JNT-2005, Shirai-Sato-JNT-2001}.
\end{rem}

\section{Lemmas}

For proving the main results, we need the following notation and lemmas.
\par
In combinatorial mathematics, the Bell polynomials of the second kind $\bell_{n,k}$ are defined by
\begin{equation*}
\bell_{n,k}(x_1,x_2,\dotsc,x_{n-k+1})=\sum_{\substack{\ell_i\in\{0\}\cup\mathbb{N}\\ \sum_{i=1}^ni\ell_i=n\\
\sum_{i=1}^n\ell_i=k}}\frac{n!}{\prod_{i=1}^{n-k+1}\ell_i!} \prod_{i=1}^{n-k+1}\Bigl(\frac{x_i}{i!}\Bigr)^{\ell_i}
\end{equation*}
for $n\ge k\ge0$. See~\cite[p.~134, Theorem~A]{Comtet-Combinatorics-74}.

\begin{lem}[{\cite[Example~2.6]{ABCLM-arXiv-2960} and~\cite[p.~136, Eq.~{[3n]}]{Comtet-Combinatorics-74}}]
The Bell polynomials of the second kind $\bell_{n,k}$ meets
\begin{multline}\label{Bell-times-eq}
\bell_{n,k}(x_1+y_1,x_2+y_2,\dotsc,x_{n-k+1}+y_{n-k+1})\\
=\sum_{r+s=k}\sum_{\ell+m=n}\binom{n}{\ell}\bell_{\ell,r}(x_1,x_2,\dotsc,x_{\ell-r+1})
\bell_{m,s}(y_1,y_2,\dotsc,y_{m-s+1}).
\end{multline}
\end{lem}

\begin{lem}[{\cite[p.~135]{Comtet-Combinatorics-74}}]
For $n\ge k\ge0$, we have
\begin{equation}\label{Bell(n-k)}
\bell_{n,k}\bigl(abx_1,ab^2x_2,\dotsc,ab^{n-k+1}x_{n-k+1}\bigr) =a^kb^n\bell_{n,k}(x_1,x_n,\dotsc,x_{n-k+1}),
\end{equation}
where $a$ and $b$ are any complex numbers.
\end{lem}

\begin{lem}[\cite{Guo-Qi-JANT-Bernoulli.tex, Zhang-Yang-Oxford-Taiwan-12}]
For $n\ge k\ge1$, we have
\begin{equation}\label{B-S-frac-value}
\bell_{n,k}\biggl(\frac12, \frac13,\dotsc,\frac1{n-k+2}\biggr)
=\frac{n!}{(n+k)!}\sum_{i=0}^k(-1)^{k-i}\binom{n+k}{k-i}S(n+i,i).
\end{equation}
\end{lem}

\begin{rem}
The special values of the Bell polynomials of the second kind $\bell_{n,k}$ are important in combinatorics and number theory. Recently, some special values for $\bell_{n,k}$ were discovered and applied in~\cite{Special-Bell2Euler.tex, Tan-Der-App-Thanks.tex, Deriv-Arcs-Cos.tex, Zhang-Yang-Oxford-Taiwan-12}.
\end{rem}

\begin{lem}\label{Chapman-lem}
Let $f(t)=1+\sum_{k=1}^\infty a_kt^k$ and $g(t)=1+\sum_{k=1}^\infty b_kt^k$ be formal power series such that $f(t)g(t)=1$. Then
\begin{equation*}
b_n=(-1)^n
\begin{vmatrix}
a_1 & 1 & 0 & 0 & \dotsm & 0\\
a_2 & a_1 & 1 & 0 & \dotsm & 0\\
a_3 & a_2 & a_1 & 1 & \dotsm & 0\\
\vdots & \vdots & \vdots & \vdots & \ddots & \vdots\\
a_{n-1} & a_{n-2} & a_{n-3} & a_{n-4} & \dotsm & 1\\
a_{n} & a_{n-1} & a_{n-2} & a_{n-3} & \dotsm & a_1
\end{vmatrix}.
\end{equation*}
\end{lem}

\begin{proof}
The identity $f(t)g(t)=1$ entails the matrix identity
\begin{equation*}
\begin{pmatrix}
1 & 0 & 0 & \dotsm & 0\\
b_1 & 1 & 0 & \dotsm & 0\\
b_2 & b_1 & 1 & \dotsm & 0\\
\vdots & \vdots & \vdots & \ddots & \vdots\\
b_n & b_{n-1} & b_{n-2} & \dotsm & 1
\end{pmatrix}
=
\begin{pmatrix}
1 & 0 & 0 & \dotsm & 0\\
a_1 & 1 & 0 & \dotsm & 0\\
a_2 & a_1 & 1 & \dotsm & 0\\
\vdots & \vdots & \vdots & \ddots & \vdots\\
a_n & a_{n-1} & a_{n-2} & \dotsm & 1
\end{pmatrix}^{-1},
\end{equation*}
where $\begin{pmatrix}\cdot\end{pmatrix}^{-1}$ stands for the inverse of an invertible matrix $\begin{pmatrix}\cdot\end{pmatrix}$. Applying Cramer's rule for a system of linear equations proves Lemma~\ref{Chapman-lem}.
\end{proof}

\begin{rem}
The idea of Lemma~\ref{Chapman-lem} was used in~\cite[pp.~22--23]{Macdonald-2nd} to express determinants of complete symmetric functions in terms of determinants of elementary symmetric functions.
\end{rem}

\section{Proofs of Theorems~\ref{BP-Stirl-thm} and~\ref{Bern-polyn-thm}}

We are now in a position to prove our main results.

\begin{proof}[Proof of Theorem~\ref{BP-Stirl-thm}]
In terms of the Bell polynomials of the second kind $\bell_{n,k}$, the Fa\`a di Bruno formula for computing higher order derivatives of composite functions is described in~\cite[p.~139, Theorem~C]{Comtet-Combinatorics-74} by
\begin{equation}\label{Bruno-Bell-Polynomial}
\frac{\td^n}{\td x^n}f\circ g(x)=\sum_{k=1}^nf^{(k)}(g(x)) \bell_{n,k}\bigl(g'(x),g''(x),\dotsc,g^{(n-k+1)}(x)\bigr).
\end{equation}
By the integral expression~\eqref{Bern-Polyn-Int-Expr}, applying the formula~\eqref{Bruno-Bell-Polynomial} to the functions $f(y)=\frac1y$ and $y=g(x)=\int_0^1 e^{x(t-u)}\td t$ results in
\begin{align*}
&\quad\frac{\td^n}{\td x^n}\biggl(\frac{xe^{ux}}{e^x-1}\biggr)
=\frac{\td^n}{\td x^n}\Biggl(\frac1{\int_0^1 e^{x(t-u)}\td t}\Biggr)\\
&=\sum_{k=1}^n\frac{(-1)^kk!}{\bigl(\int_0^1 e^{x(t-u)}\td t\bigr)^{k+1}} \bell_{n,k}\biggl(\int_0^1(t-u)e^{x(t-u)}\td t, \\
&\quad\int_0^1 (t-u)^2e^{x(t-u)}\td t,\dotsc, \int_0^1 (t-u)^{n-k+1}e^{x(t-u)}\td t\biggr)\\
&\to \sum_{k=1}^n(-1)^kk! \bell_{n,k}\biggl(\int_0^1(t-u)\td t, \int_0^1(t-u)^2\td t,\dotsc,\int_0^1(t-u)^{n-k+1}\td t\biggr)\\
&=\sum_{k=1}^n(-1)^kk! \bell_{n,k}\biggl(\frac{(1-u)^2-(-u)^2}{2}, \frac{(1-u)^3-(-u)^3}{3},\\
&\quad\dotsc,\frac{(1-u)^{n-k+2}-(-u)^{n-k+2}}{n-k+2}\biggr)
\end{align*}
as $x\to0$. Further employing~\eqref{Bell-times-eq}, \eqref{Bell(n-k)}, and~\eqref{B-S-frac-value} acquires
\begin{multline*}
\begin{aligned}
\frac{\td^n}{\td x^n}\biggl(\frac{xe^{ux}}{e^x-1}\biggr)\bigg|_{x=0}
&=\sum_{k=1}^n(-1)^kk! \sum_{r+s=k}\sum_{\ell+m=n}\binom{n}{\ell}\\
&\quad\times\bell_{\ell,r}\biggl(\frac{(1-u)^2}{2}, \frac{(1-u)^3}{3}, \dotsc,\frac{(1-u)^{\ell-r+2}}{\ell-r+2}\biggr)\\
&\quad\times\bell_{m,s}\biggl(-\frac{(-u)^2}{2}, -\frac{(-u)^3}{3}, \dotsc,-\frac{(-u)^{m-s+2}}{m-s+2}\biggr)
\end{aligned}\\
\begin{aligned}
&=\sum_{k=1}^n(-1)^kk! \sum_{r+s=k}\sum_{\ell+m=n}\binom{n}{\ell}(1-u)^{\ell+r} \bell_{\ell,r}\biggl(\frac1{2}, \frac1{3}, \dotsc,\frac1{\ell-r+2}\biggr)\\
&\quad\times u^s(-u)^m\bell_{m,s}\biggl(\frac1{2}, \frac1{3}, \dotsc,\frac1{m-s+2}\biggr)
\end{aligned}\\
\begin{aligned}
&=\sum_{k=1}^nk! \sum_{r+s=k}\sum_{\ell+m=n}(-1)^m\binom{n}{\ell} \frac{\ell!}{(\ell+r)!}\frac{m!}{(m+s)!}\\
&\quad\times\Biggl[\sum_{i=0}^r\sum_{j=0}^s(-1)^{i+j}\binom{\ell+r}{r-i}
\binom{m+s}{s-j}S(\ell+i,i)S(m+j,j)\Biggr]u^{m+s}(1-u)^{\ell+r}.
\end{aligned}
\end{multline*}
As a result, the formula~\eqref{BP-Stirl-form} follows immediately.
\par
Letting $u=0$ in~\eqref{BP-Stirl-form}, simplifying, and interchanging the order of sums lead to the formula~\eqref{Bernoulli-Stirling-formula}.
The proof of Theorem~\ref{BP-Stirl-thm} is complete.
\end{proof}

\begin{proof}[The first proof of Theorem~\ref{Bern-polyn-thm}]
Let $u=u(z)$ and $v=v(z)\ne0$ be differentiable functions. In~\cite[p.~40]{Bourbaki-Spain-2004}, the formula
\begin{equation}\label{Sitnik-Bourbaki}
\frac{\td^k}{\td z^k}\biggl(\frac{u}{v}\biggr)
=\frac{(-1)^k}{v^{k+1}}
\begin{vmatrix}
u & v & 0 & \dots & 0\\
u' & v' & v & \dots & 0\\
u'' & v'' & 2v' & \dots & 0\\
\hdotsfor[2]{5}\\
u^{(k-1)} & v^{(k-1)} & \binom{k-1}1v^{(k-2)} &  \dots & v\\
u^{(k)} & v^{(k)} & \binom{k}1v^{(k-1)} & \dots & \binom{k}{k-1}v'
\end{vmatrix}
\end{equation}
for the $k$th derivative of the ratio $\frac{u(z)}{v(z)}$ was listed. For easy understanding and convenient availability, we now reformulate the formula~\eqref{Sitnik-Bourbaki} as
\begin{equation}\label{Sitnik-Bourbaki-reform}
\frac{\td^k}{\td z^k}\biggl(\frac{u}{v}\biggr)
=\frac{(-1)^k}{v^{k+1}}
\begin{vmatrix}
A_{(k+1)\times1}&B_{(k+1)\times k}
\end{vmatrix}_{(k+1)\times(k+1)},
\end{equation}
where the matrices
\begin{equation*}
A_{(k+1)\times1}=(a_{\ell,1})_{0\le \ell\le k}
\end{equation*}
and
\begin{equation*}
B_{(k+1)\times k}=(b_{\ell,m})_{0\le \ell\le k,0\le m\le k-1}
\end{equation*}
satisfy
\begin{equation*}
a_{\ell,1}=u^{(\ell)}(z)\quad \text{and}\quad b_{\ell,m}=\binom{\ell}{m}v^{(\ell-m)}(z)
\end{equation*}
under the conventions that $v^{(0)}(z)=v(z)$ and that $\binom{p}{q}=0$ and $v^{(p-q)}(z)\equiv0$ for $p<q$. See also~\cite[Section~2.2]{Tan-Der-App-Thanks.tex} and~\cite[Lemma~2.1]{Euler-No-3Sum.tex}.
By the integral expression~\eqref{Bern-Polyn-Int-Expr}, applying the formula~\eqref{Sitnik-Bourbaki-reform} to $u(z)=1$ and $v(z)=\int_0^1 e^{z(t-u)}\td t$ yields $a_{1,1}=1$, $a_{\ell,1}=0$ for $\ell>1$,
\begin{align*}
b_{\ell,m}&=\binom{\ell}{m}\int_0^1 (t-u)^{\ell-m}e^{z(t-u)}\td t\\
&\to \binom{\ell}{m}\int_0^1 (t-u)^{\ell-m}\td t, \quad z\to0\\
&=\binom{\ell}{m}\frac{(1-u)^{\ell-m+1}-(-u)^{\ell-m+1}}{\ell-m+1}
\end{align*}
for $0\le \ell\le k$ and $0\le m\le k-1$ with $\ell\ge m$, and
\begin{align*}
\frac{\td^k}{\td z^k}\biggl(\frac{ze^{uz}}{e^z-1}\biggr)
&=\frac{(-1)^k}{b_{0,0}^{k+1}} |b_{\ell,m}|_{1\le \ell\le k,0\le m\le k-1}\\
&\to(-1)^k \biggl|\binom{\ell}{m}\int_0^1 (t-u)^{\ell-m}\td t\biggr|_{1\le \ell\le k,0\le m\le k-1}, \quad z\to0\\
&=(-1)^k\biggl|\binom{\ell}{m}\frac{(1-u)^{\ell-m+1}-(-u)^{\ell-m+1}}{\ell-m+1}\biggr|_{1\le \ell\le k,0\le m\le k-1}.
\end{align*}
The formula~\eqref{Bern-Polyn-determ} is proved.
\par
The formula~\eqref{Bern-No-determ} follows readily from taking $u=0$ in~\eqref{Bern-Polyn-determ}. The first proof of Theorem~\ref{Bern-polyn-thm} is complete.
\end{proof}

\begin{proof}[The second proof of Theorem~\ref{Bern-polyn-thm}]
Applying Lemma~\ref{Chapman-lem} to the function $g(t)=\frac{te^{ut}}{e^t-1}$ and $f(t)=\frac{e^{(1-u)t}-e^{-ut}}{t}$ reveals that $b_n=\frac{B_n(u)}{n!}$ and $a_n=\frac{(1-u)^{n+1}-(-u)^{n+1}}{n!}$. Hence, by virtue of Lemma~\ref{Chapman-lem},
\begin{equation*}
B_n(u)=(-1)^nn!
\begin{vmatrix}
\frac{(1-u)^{\ell-m+1}-(-u)^{\ell-m+1}}{(\ell-m+1)!}
\end{vmatrix}_{1\le\ell\le n,0\le m\le n-1}.
\end{equation*}
Multiplying the row $\ell$ of this determinant by $\ell!$ and dividing the row $m$ by $m!$ gives
\begin{equation*}
B_n(u)=(-1)^n\biggl|\frac1{\ell+1}\binom{\ell+1}{m} \bigl[(1-u)^{\ell-m+1}-(-u)^{\ell-m+1}\bigr]\biggr|_{1\le \ell\le n,0\le m\le n-1}.
\end{equation*}
The formula~\eqref{Bern-Polyn-determ} is thus proved. The second proof of Theorem~\ref{Bern-polyn-thm} is complete.
\end{proof}

\begin{rem}
This manuscript is a revision and extension of the first version of the preprint~\cite{Exp-Bernoulli-Polynomials.tex}.
\end{rem}

\subsection*{Acknowledgements}
The first author thanks Professor J. F. Peters at University of Manitoba for his recommendation of the reference~\cite{ABCLM-arXiv-2960} on 26 May 2015 through the ResearchGate website.
\par
This manuscript was drafted while the first author visited Kyungpook National University and Hannam University between 25--31 May 2015 to attend the 2015 Conference and Special Seminar of the Jangjeon Mathematical Society. The first author appreciates Professors Dmitry V. Dolgy and Taekyun Kim at Kwangwoon University, Byung-Mun Kim at Dongguk University, Dae San Kim at Sogang University, Seog-Hoon Rim at Kyungpook National University, Jong Jin Seo at Pukyong National University, and many others for their invitation, support, and hospitality.


\begin{thebibliography}{99}

\bibitem{ABCLM-arXiv-2960}
A. Aboud, J.-P. Bultel, A. Chouria, J.-G. Luque, O. Mallet, \textit{Bell polynomials in combinatorial Hopf algebras}, available online at \url{http://arxiv.org/abs/1402.2960}.

\bibitem{Bourbaki-Spain-2004}
N. Bourbaki, \emph{Functions of a Real Variable, Elementary Theory}, Translated from the 1976 French original by Philip Spain. Elements of Mathematics (Berlin). Springer-Verlag, Berlin, 2004; Available online at \url{http://dx.doi.org/10.1007/978-3-642-59315-4}.

\bibitem{Comtet-Combinatorics-74}
L. Comtet, \emph{Advanced Combinatorics: The Art of Finite and Infinite Expansions}, Revised and Enlarged Edition, D. Reidel Publishing Co., Dordrecht and Boston, 1974.

\bibitem{Gould-MAA-1972}
H. W. Gould, \emph{Explicit formulas for Bernoulli numbers}, Amer. Math. Monthly \textbf{79} (1972), 44\nobreakdash--51.

\bibitem{emv-log-convex-simple.tex}
B.-N. Guo and F. Qi, \textit{A simple proof of logarithmic convexity of extended mean values}, Numer. Algorithms \textbf{52} (2009), no.~1, 89\nobreakdash--92; Available online at \url{http://dx.doi.org/10.1007/s11075-008-9259-7}.

\bibitem{Guo-Qi-JANT-Bernoulli.tex}
B.-N. Guo and F. Qi, \textit{An explicit formula for Bernoulli numbers in terms of Stirling numbers of the second kind}, J. Anal. Number Theory \textbf{3} (2015), no.~1, 27\nobreakdash--30; Available online at \url{http://dx.doi.org/10.12785/jant/030105}.

\bibitem{Special-Bell2Euler.tex}
B.-N. Guo and F. Qi, \textit{Explicit formulas for special values of the Bell polynomials of the second kind and the Euler numbers}, ResearchGate Technical Report, available online at \url{http://dx.doi.org/10.13140/2.1.3794.8808}.

\bibitem{ijmest-bernoulli}
B.-N. Guo and F. Qi, \textit{Generalization of Bernoulli polynomials}, Internat. J. Math. Ed. Sci. Tech. \textbf{33} (2002), no.~3, 428\nobreakdash--431; Available online at \url{http://dx.doi.org/10.1080/002073902760047913}.

\bibitem{Guo-Qi-Filomat-2011-May-12.tex}
B.-N. Guo and F. Qi, \textit{The function $(b^x-a^x)/x$: Logarithmic convexity and applications to extended mean values}, Filomat \textbf{25} (2011), no.~4, 63\nobreakdash--73; Available online at \url{http://dx.doi.org/10.2298/FIL1104063G}.

\bibitem{Jeong-Kim-Son-JNT-2005}
S. Jeong, M.-S. Kim, and J.-W. Son, \emph{On explicit formulae for Bernoulli numbers and their counterparts in positive characteristic}, J. Number Theory \textbf{113} (2005), no.~1, 53\nobreakdash--68; Available online at~\url{http://dx.doi.org/10.1016/j.jnt.2004.08.013}.

\bibitem{Macdonald-2nd}
I. G. Macdonald, \textit{Symmetric Functions and Hall Polynomials}, 2nd ed., With contributions by A. Zelevinsky, Oxford Mathematical Monographs, Oxford Science Publications, The Clarendon Press, Oxford University Press, New York, 1995.

\bibitem{Tan-Der-App-Thanks.tex}
F. Qi, \textit{Derivatives of tangent function and tangent numbers}, available online at \url{http://arxiv.org/abs/1202.1205}.

\bibitem{Exp-Bernoulli-Polynomials.tex}
F. Qi, \textit{Two closed forms for the Bernoulli polynomials}, available online at \url{http://arxiv.org/abs/1506.02137}.

\bibitem{Qi-Springer-2012-Srivastava.tex}
F. Qi, Q.-M. Luo, and B.-N. Guo, \textit{The function $(b^x-a^x)/x$: Ratio's properties}, In: \emph{Analytic Number Theory, Approximation Theory, and Special Functions}, G. V. Milovanovi\'c and M. Th. Rassias (Eds), Springer, 2014, pp.~485\nobreakdash--494; Available online at \url{http://dx.doi.org/10.1007/978-1-4939-0258-3_16}.

\bibitem{Euler-No-3Sum.tex}
F. Qi and C.-F. Wei, \textit{Several closed expressions for the Euler numbers}, ResearchGate Technical Report, available online at \url{http://dx.doi.org/10.13140/2.1.3474.7688}.

\bibitem{jmaa-ii-97}
F. Qi and S.-L. Xu, \textit{Refinements and extensions of an inequality, II}, J. Math. Anal. Appl. \textbf{211} (1997), no.~2, 616\nobreakdash--620; Available online at \url{http://dx.doi.org/10.1006/jmaa.1997.5318}.

\bibitem{(b^x-a^x)/x}
F. Qi and S.-L. Xu, \textit{The function $(b^x-a^x)/x$: Inequalities and properties}, Proc. Amer. Math. Soc. \textbf{126} (1998), no.~11, 3355\nobreakdash--3359; Available online at \url{http://dx.doi.org/10.1090/S0002-9939-98-04442-6}.

\bibitem{Deriv-Arcs-Cos.tex}
F. Qi and M.-M. Zheng, \textit{Explicit expressions for a family of the Bell polynomials and applications}, Appl. Math. Comput. \textbf{258} (2015), 597\nobreakdash--607; Available online at \url{http://dx.doi.org/10.1016/j.amc.2015.02.027}.

\bibitem{Shirai-Sato-JNT-2001}
S. Shirai and K.-I. Sato, \emph{Some identities involving Bernoulli and Stirling numbers}, J. Number Theory \textbf{90} (2001), no.~1, 130\nobreakdash--142; Available online at \url{http://dx.doi.org/10.1006/jnth.2001.2659}.

\bibitem{Zhang-Yang-Oxford-Taiwan-12}
Z.-Z. Zhang and J.-Z. Yang, \emph{Notes on some identities related to the partial Bell polynomials}, Tamsui Oxf. J. Inf. Math. Sci. \textbf{28} (2012), no.~1, 39\nobreakdash--48.

\end{thebibliography}
\end{document}